\documentclass[12pt]{amsart}

\usepackage{amsmath,amsthm}
\usepackage{amssymb}
\usepackage{latexsym}
\usepackage{graphicx}

\usepackage[usenames,dvipsnames]{color}

\usepackage{tikz}
\usetikzlibrary{shapes.geometric}
\usetikzlibrary{fit}

\newtheorem{theorem}{Theorem}[section]
\newtheorem{lemma} [theorem] {Lemma}%[section]
\let\eps=\epsilon
\def\enddiscard{}
\long\def\discard#1\enddiscard{}

\newcommand{\bin}{{\rm Bin}}

\newcommand{\E}{\mathbb E}
\newcommand{\pr}{\mathbb P}

\newcommand{\inu}{\in_{u}}
\newcommand{\whp}{whp }
\newcommand{\m}[1]{\marginpar{\tiny{#1}}}

\newcommand{\ca}{{\mathcal A}}
\newcommand{\cA}{{\mathcal A}}

\newcommand{\cc}{{\mathcal C}}
\newcommand{\cC}{{\mathcal C}}

\newcommand{\ce}{{\mathcal E}}
\newcommand{\cE}{{\mathcal E}}

\newcommand{\cg}{{\mathcal G}}
\newcommand{\cG}{{\mathcal G}}
\newcommand{\cp}{{\mathcal P}}
\newcommand{\cP}{{\mathcal P}}

\newcommand{\ct}{{\mathcal T}}
\newcommand{\cT}{{\mathcal T}}
\newcommand{\cu}{{\mathcal U}}

\newcommand{\ex}{{\rm Ex}}
\newcommand{\bt}{{\rm BT}}
\newcommand{\btf}{{\rm BF}}

\newcommand{\dist}{{\rm dist}}

\begin{document}

\title{Random graphs from a block-stable class}

%\date{}  
\date{\today}

\author{Colin McDiarmid and Alexander Scott}

\address{${}^1$
Department of Statistics
1 South Parks Road
Oxford, OX1 3TG,
United Kingdom; 
${}^2$
Mathematical Institute,
University of Oxford,
Andrew Wiles Building,
Radcliffe Observatory Quarter,
Woodstock Road,
Oxford,
OX2 6GG, UK}

\maketitle

\begin{abstract}
A class of graphs is called {\em block-stable} when a graph is in the class if and only if each of its blocks is.
We show that, as for trees, for most $n$-vertex graphs in such a class, each vertex is in at most $(1+o(1)) \log n / \log\log n$ blocks, and each path passes through at most $5 (n \log n)^{1/2}$ blocks.
These results extend to `weakly block-stable' classes of graphs.
\end{abstract}

%%%%%%%%%%%%%%%%%%%%%%%%%%%%%%%%%%%%%%%%%%%%%%%%%%%%%%%%%%%%%%%%%%

\section{Introduction}
\label{sec.intro}

  A {\em block} in a graph is a maximal 2-connected subgraph or the subgraph formed by a bridge or an isolated vertex. (A \emph{bridge} is an edge the deletion of which increases the number of components.)
Call a class of graphs (always assumed to be closed under isomorphism) 
%\m{closed under isomorphism - respond to ref?}
{\em block-stable} when a graph $G$ is in the class if and only if each block of $G$ is in the class.
For example, the class of all forests is block-stable and more generally so is any minor-closed class of graphs with 2-connected excluded minors.  A different example is the class of all graphs in which each block is a triangle.

In this paper, we are interested in typical properties of graphs from such a class.
%Our results in fact hold more generally, for `weakly block-stable' sets of graphs.  
Indeed we are interested in more general classes of graphs, namely `weakly block-stable' classes.
To define this notion, let us first introduce an equivalence relation on (finite) graphs, which is natural in this context.  Given connected graphs $G$ and $H$, let $G \sim H$ if they have the same vertex set and the same number of blocks of each kind (up to isomorphism). 
%In this case $G$ and $H$ must have the same number of vertices.
Given general graphs $G$ and $H$, let $G \sim H$ if we can list the components as $G_1,\ldots,G_k$ and $H_1,\ldots,H_k$ (for some $k$) so that $G_i \sim H_i$ for each~$i$.  %Again, if $G \sim H$ then $G$ and $H$ must have the same numbers of vertices (and edges and blocks). 
%
%Now that we have introduced the equivalence relation $\sim$, %(the only one we shall mention in this paper), 
We say that a class $\cA$ of graphs is \emph{weakly block-stable} if whenever $G \in \cA$ and $H \sim G$ then $H \in \cA$.
Clearly a block-stable class is weakly block-stable, but not conversely.  %\m{remove or move: For example, the class of connected graphs in weakly block-stable class is weakly block-stable.}  
%(It is natural to insist that our graph classes $\cA$ are closed under isomorphism, since for example if $\cA$ contains the graph $G$ on vertex set $\{1,\ldots,5\}$ formed by overlapping a triangle on $\{1,2,3\}$ and a triangle on $\{3,4,5\}$, we want $\cA$ also to contain each graph on $\{1,\ldots,5\}$ formed from two triangles overlapping in a single vertex.) 
\smallskip

As mentioned above, we are most interested in typical properties of graphs from a block-stable class, but our results extend to weakly block-stable classes of graphs, and indeed that is the natural context for our investigations.
In particular, we are interested in the maximum number of blocks containing a given vertex, and the maximum number of blocks a path can pass through.

For a connected graph $G$, these are essentially properties of the {\em block tree} $\bt(G)$ of $G$, which is the bipartite graph with a node $x_v$ for each vertex $v$ and a node $y_B$ for each block $B$, where $x_v$ and $y_B$ are adjacent if and only if $v \in B$. (There is an alternative slimmer version of the block tree, in which vertices which are not cut-vertices are ignored.)
If $G$ is not necessarily connected, we let the {\em block forest} $\btf(G)$ be the disjoint union of the block trees of the components.
%we shall not use that version.)

Given a set $\cA$ of graphs,
%, which is {\em non-trivial}, that is some graph in $\ca$ has at least one edge.
for each positive integer $n$ let $\cA_n$ denote the set of graphs in $\cA$ on vertex set $[n]:=\{1,\ldots,n\}$.  Also, let $R_n \inu \cA$ mean that $R_n$ is sampled uniformly from $\cA_n$.  When we use this notation we implicitly consider only integers $n$ such that $\cA_n$ is non-empty.
Now suppose that $\cA$ is weakly block-stable and $\cP$ is any graph property.
Note that $\cA_n$ may be partitioned into the distinct equivalence classes $[G]$ for $G \in \cA_n$ (where the equivalence relation is graph isomorphism).
Thus if we can show for each $G \in \cA_n$ that
$\pr(R \in \cP) \geq t$ when $R \in_u [G]$, then it will follow that $\pr(R_n \in \cP) \geq t$ when $R_n \in_u \cA$.
We say that a sequence $(E_n)$ of events holds \emph{with high probability} (\whp\!\!) if $\pr(E_n) \to 1$ %the probability that $E_n$ holds $\to 1$ 
as $n \to \infty$.

Let $\cT$ denote the class of trees, and let $T_n \inu \cT$. It will be natural for us to compare the block tree $\bt(R_n)$ with $T_n$,
%\m{added para; def OK? ADS}
and to compare the associated degree sequences.  Given two random variables $X$ and $Y$, we say that $X$ is {\em stochastically at most $Y$} if
$\pr(X\geq t)\le\Pr(Y\geq t)$ for every real number $t$.
%; or, equivalently, if there is a coupling of $X$ and $Y$ such that $X\le Y$ with probability 1.
More generally, for two sequences ${\bf X}=(X_1,\dots,X_n)$ and ${\bf Y}=(Y_1,\dots,Y_n)$ of random variables, we say that {\em ${\bf X}$ is stochastically at most ${\bf Y}$} if
$\E[f({\bf X})] \leq \E[f({\bf Y})]$ for each non-decreasing integrable real-valued function $f$ on ${\mathbb R}^n$.
% or, equivalently, \m{check} if there is a coupling such that, with probability 1, we have $X_i\le Y_i$ for all $i$.
%\m{I dropped the coupling version - ok?}

We are interested in typical properties of $R_n$ when $\cA$ is a (weakly) block-stable class, or is the set of connected graphs in such a class; and in particular we focus on degrees of nodes $x_v$ and on long paths in $\bt(R_n)$ or $\btf(R_n)$.  We present our main results in the next two subsections.
% are robust general inequalities, which hold for weakly block-stable sets of graphs.  
%example even when we specify precisely how often each possible block appears.

Consider briefly a related but distinct setting, where there are results of a different nature. Suppose that our block-stable class is the class of all series-parallel graphs or another `subcritical' graph class, or it is the class of planar graphs, or another such class where we know the corresponding generating functions suitably well.  In such cases, we may be able to deduce precise asymptotic results, for example about vertex degrees or the numbers and sizes of blocks, by using analytic techniques or by analysing Boltzmann samplers:
see for example~\cite{bps09}, \cite{dfkkr}, \cite{dgnps2015}, \cite{dn2013}, \cite{fp11}, \cite{gmn13}, \cite{gn09a}, \cite{gn09b}, \cite{gnr}, \cite{ps2010}, \cite{psw2015+},
and for an authoritative recent overview of related work on random planar graphs and beyond see the article~\cite{noy2014} by Marc Noy.
%(and see references in~\cite{gmn13} for further papers on maximum degrees) and see also~\cite{bfss01} (concerning maps.)
%\m{add refs for degrees?}
%see \cite{bps09}, Panagiotou and Steger~\cite{ps2010}, Fountoulakis and Panagiotou~\cite{fp11}, \cite{dfkkr}, Panagiotou, Stuffler and Weller.) 

The main tools we use in our proofs are a tree-like graph $\widetilde{G}$ related to the block-tree of a graph $G$, and a corresponding tree $T_G$, together with a slight extension of Pr$\ddot{\rm u}$fer coding: these are discussed in the next section.  The proofs of Theorem~\ref{thm.degreesnew} and Theorem~\ref{thm.diamnew} are completed in Sections~\ref{sec.blockdeg} and \ref{sec.pathlengths} respectively, and we make some brief concluding remarks in Section~\ref{sec.concl}.

%%%%%%%%%%%%%%%%%%%%%%%%%%%%%%%%%%%%%%%%%%%%%%%%%%%%%%%%%%%%%%%%%%%%%%%%

\subsection{Block-degrees of vertices}

First consider the number of blocks in $G$ containing a vertex $v$, that is, the degree of the node $x_v$ in the block tree $\bt(G)$: let us call this number the \emph{block-degree} of $v$, and denote it by $\widetilde{d}_G(v)$. Observe that if $G$ is a tree (with at least two vertices) then $\widetilde{d}_G(v)$ is just the degree $d_G(v)$ of $v$ in $G$.
Denote the maximum of the numbers $\widetilde{d}_G(v)$ by $\widetilde{\Delta}(G)$. Recall that, for $T_n \inu \ct$, the maximum degree $\Delta$
satisfies 
\begin{equation}\label{treeD1}
  \Delta(T_n) \sim \log n / \log\log n \;\;\; \whp ,
\end{equation}
see~\cite{moon68},~\cite{cgs94}. %\m{and \cite{abt03}? check.  discard?}
Also, for any constant $c>0$, 
\begin{equation}\label{treeD2}
  \pr(\Delta(T_n) \geq c n / \log n) = e^{-(c+o(1))n}. 
\end{equation}
Both these results follow easily from considering Pr$\ddot{\rm u}$fer coding.
%\m{ref? amplify?}

The following theorem says roughly that block degrees are no larger than those for a random tree $T_n$. In particular, if we sample $R_n$ uniformly from the connected graphs in a block-stable class, then the maximum block degree $\widetilde{\Delta}(R_n)$ is stochastically at most $\Delta(T_n)$, and so \whp it is no more than about $\log n / \log\log n$;
%Further, the same result holds if we fix an $n$-vertex connected graph $G$ in the class, and sample uniformly from the connected graphs with the same blocks as $G$; 
and indeed we can improve the bound if there are few blocks.

\begin{theorem} \label{thm.degreesnew}
  Let $\cA$ be a weakly block-stable class of graphs and let $\cC$ be the class of connected graphs in $\cA$.

  (a) For $R_n \inu \cC$, the list of block degrees $(\widetilde{d}_{R_n}(v): v \in [n])$ is stochastically at most $(d_{T_n}(v):v \in [n])$, where $T_n \inu \cT$ is a uniformly random tree on~$[n]$;
and in particular the maximum block degree $\widetilde{\Delta}(R_n)$ is stochastically at most $\Delta(T_n)$.
%These results continue to hold if we condition on the number of blocks, and on the list of blocks appearing.

(b) For $R_n \inu \cA$,
\begin{equation} \label{eqn.deg1} 
  \widetilde{\Delta}(R_n) \leq (1+\eps(n))\log n / \log\log n \;\;\; \whp
\end{equation}
where $\eps(n)=o(1)$, and indeed we may take $\eps(n)= 2 \log\log\log n / \log\log n$, (whatever $\cA$ is); and for any constant $c>0$
\begin{equation} \label{eqn.deg2}
 \pr( \widetilde{\Delta}(R_n) \geq c n/\log n) \leq e^{-(1-\eta(n))c n},
\end{equation}
where $\eta(n)=o(1)$, and indeed we may take $\eta(n)= 2 \log\log n / \log n$. 
%where the $o(1)$ terms depend only on $n$ (and %not on the block class).\m{added ADS}
% = e^{-\Omega(n)}.\] 
%The inequality~\eqref{eqn.deg1} %and \eqref{eqn.deg2}
%These results continue to hold if we condition on the numbers of components and blocks, and on the list of blocks appearing; and
Further, if the number of blocks in graphs in $\cA_n$ is at most $k=k(n)$ where $k \to \infty$ as $n \to \infty$, then
\begin{equation} \label{eqn.degk} 
  \widetilde{\Delta}(R_n) \leq (1+\eps(k))\log k / \log\log k \;\;\; \whp
\end{equation} 
where the function $\eps$ is as above.
%where the $o(1)$ term depends only on $k$ (and %not on the block class).\m{added ADS}
%in $R_n$; and for each component on the number   %$n_1,\ldots,n_{\kappa}$ of vertices, on the number 
%$k$ of blocks in $R_n$ and on the list of blocks $B_1,\ldots,B_k$. 
\end{theorem}

%\noindent
%When we condition on an event as in this theorem, we assume implicitly that the conditioning event has positive probability. Thus in part (b), if there are $j$ components and $k$ blocks and the list of blocks is $B_1,\ldots,B_k$, then we are assuming that \[ \sum_{i=1}^{k} (v(B_i)-1) = n -j. \] 

\noindent
For $R_n \inu \ca$ as in part (b), there is no detailed result on stochastic dominance by a tree like that for $R_n \inu \cc$ in part (a) (see the comment following Lemma~\ref{lem.maxdegreenew} below). %\m{ by a tree}
Of course the inequality~(\ref{eqn.degk}) implies the earlier inequality~(\ref{eqn.deg1}) since there can be at most $n$ blocks. %\m{was $n-1$}
Theorem \ref{thm.degreesnew} will be deduced from more precise non-asymptotic results, Lemmas~\ref{lem.conn-degree} and~\ref{lem.maxdegreenew} below.

%\m{NEW}
Finally here, consider the class $\ex(C_4)$ of graphs with no minor the cycle $C_4$ on 4 vertices.  For graphs in this class, each block is a vertex or an edge or a triangle.  Thus, for $R_n \inu \ex(C_4)$, by~(\ref{eqn.deg1}) we have
\[ \Delta(R_n) \leq (2+o(1))\log n / \log\log n \;\;\; \whp, \]
as in~\cite{gmn13} Lemma 10.  This inequality is tight, and we have 
\[  \Delta(R_n) \log\log n / \log n \to 2 \; \mbox{ in probability as } n \to \infty.  \]
For the lower bound, see Theorem 4.1 of~\cite{mr08} (suitably amended) or Theorem 3 part 2 of~\cite{gmn13}. 
 
%%%%%%%%%%%%%%%%%%%%%%%%%%%%%%%%%%%%%%%%%%%%%%%%%%%%%%%%%%%%%%%%%%%%%%%%%%%%%%%%%%%%  

\subsection{Block length of paths}

Now we consider paths, and see that graphs in $\cA$ are unlikely to contain any path which passes through many blocks (that is, any path which has edges in many different blocks). The
%distance between two vertices in a graph is the least number of edges in a path between them; and the
{\em diameter} of a graph is the maximum distance %(in the block forest)
between any two vertices in the same component.

For $T_n \inu \cT$, with probability near 1 the diameter of $T_n$ is of order $\sqrt{n}$~\cite{rs67}:
% $\Theta(\sqrt{n \log n})$, see~\cite{}. 
%\m{$\Theta(\sqrt{n \log n})$ true? NO!}
more exactly, for any $\eps>0$ there are constants $0<c_1<c_2$ such that with probability at least $1-\eps$ the diameter is between $c_1 \sqrt{n}$ and $c_2 \sqrt{n}$.
See~\cite{fo82} for a precise result on the maximum length of a path from a root vertex to another vertex (see also Theorem 4.8 of~\cite{drmotabook}).  For contrast, it was shown in~\cite{dn2013} that whp %given $\eps>0$ 
the diameter of a random planar graph $R_n$ is $n^{\frac14 +o(1)}$,
% of order $\sqrt{n}$ lies \whp in the interval $(n^{\frac14 -\eps}, n^{\frac14 +\eps})$, see~\cite{cfgn}, 
see also~\cite{cfgn} for more precise information. Also, observe that for $n \geq 2$ the probability that $T_n$ has diameter $n\!-\!1$ (that is, $T_n$ is a path) is
$\; n!/(2n^{n-2}) = e^{-n +O(\log n)}$.

The following theorem shows in particular that, if we sample $R_n$ uniformly from the connected graphs in a block-stable class, then \whp the block tree $\bt(R_n)$ has diameter at most $5 \sqrt{n \log n}$.
% Further, the same result holds if we fix an $n$-vertex connected graph $G$ in the class, and sample uniformly from the connected graphs with the same blocks as~$G$. 
We conjecture that the extra factor $\sqrt{\log n}$ (compared with the random tree $T_n$) could be replaced by any function tending to $\infty$.  

\begin{theorem} \label{thm.diamnew}
  Let $\cA$ be a weakly block-stable class of graphs, and let $R_n \inu \cA$. % $\cc$ be the class of connected graphs in $\ca$. 
  Then %(provided these sets are non-empty) 
  \whp the block forest $\btf(R_n)$ has diameter at most $5 \sqrt{n \log n}$;
and for each $\eps>0$ the probability that $\btf(R_n)$ has diameter at least $\eps n$ is $e^{-\Omega(n)}$, where the function $\Omega(n)$ does not depend on %s only on $\eps$ (and not on 
the class $\cA$.
%
%These results continue to hold if we condition on the numbers of components and blocks, and on the list of blocks appearing; and
Further, if the number of blocks in graphs in $\cA_n$ is at most $k=k(n)$ where $k \to \infty$ as $n \to \infty$, then \whp the block forest $\btf(R_n)$ has diameter at most $5 \sqrt{k \log k}$.
%in $R_n$; and for each component $i=1,\ldots, \kappa$
%on the number   %$n_1,\ldots,n_{\kappa}$  $n_i$ of vertices, on the number 
%$k$ of blocks and on the list of blocks $B_1,\ldots,B_k$. %$B_i^1,\ldots,B_i^{k_i}$.   
\end{theorem}
Observe that if the blocks in the graphs considered are of bounded size (for example in the block class $\ex(C_4)$ of graphs with no minor $C_4$ each block has at most 3 vertices) then these results transfer easily from block trees or forests to $R_n$ itself.

  To prove Theorem~\ref{thm.diamnew} we give a precise non-asymptotic lemma, Lemma~\ref{lem.pathnew} %in Section~\ref{sec.pathlengths} 
below, from which the theorem %Theorem~\ref{thm.diamnew}
will follow easily.

%%%%%%%%%%%%%%%%%%%%%%%%%%%%%%%%%%%%%%%%%%%%%%%%%%%%%%%%%%%%%%%%%%

\section{Trees and coding following Pr\"ufer}
\label{sec.Pcoding}

Let $\cC$ be the class of connected graphs in a weakly block-stable class; or equivalently, let $\cC$ be a weakly block-stable class of connected graphs.  With respect to the equivalence relation we introduced earlier,
$\cC_n$ is naturally partitioned into equivalence classes $[G]$.
 We shall show that $\cC_n$ may be partitioned more finely into parts $\cG$, so that if $G \in \cC$ has $k$ blocks and $\cG$ is a part contained in $[G]$,
% so in particular in each part $\cG$ each graph has the same family 
%  \m{ was `contains the graphs with a specific family'} of $k$ blocks (up to isomorphism), and 
then there is a bijection between $\cG$ and $[n]^{k-1}$, similar to that in Pr$\ddot{\rm u}$fer coding, see for example the book by van Lint and Wilson~\cite{vw01}.
%  \m{ref for Pr$\ddot{\rm u}$fer coding? Stanley? West? Van Lint and Wilson, Matousek and Nesetril}
The encoding that we use here is essentially the same as that introduced by Kajimoto \cite{k03}. %\m{ref here?}

Given a connected graph $G$ on vertex set $V=[n]$ with $k$ blocks, we will `explode' $G$ into a tree-like graph $\widetilde{G}$ rooted at vertex $n$. The graph $\widetilde G$ will contain vertex-disjoint copies of the blocks of $G$ (plus one additional root vertex), joined together in a tree structure that indicates how the blocks are joined together in $G$.
See the graphs $H$ and $\tilde{H}$ in Figure 1.

Informally, $\widetilde G$ is constructed as follows: we begin by taking vertex-disjoint copies $B_1,\dots,B_k$ of the $k$ blocks of $G$ and add an extra block containing the single vertex $n$.
%  \m{was $B_{k+1}=\{n\}$}  
Thus we get one copy of a vertex for each block it belongs to, and an additional copy of $n$ (so a vertex appears more than once if and only if it is a cutvertex or $n$).  
For each vertex $j$ that occurs more than once, if $B$ is the block containing $j$ which is nearest to vertex $n$ in $G$, 
%we let $v_i$ be the copy of $i$ in the block that lies closest to $n$ in $G$.  
%later v_i refers to block i not vertex i
we give new labels to the copies of $j$ other than the copy in $B$, and refer to these as `ghost' vertices; and finally, we join vertex $j$ to its corresponding ghost vertices. 
%\m{See Figure 1 for an example of a graph $H$ and $\tilde{H}$?  Earlier?}

More formally, we apply the following procedure:   
\begin{itemize}
\item 
For each block $B$ of $G$, let $v_B$ be the vertex $n$ if $n$ is in $B$, and otherwise let $v_B$ be the cut-vertex in $B$ which separates $B$ from vertex~$n$.
%\m{was `for the block $B$ that contains $n$, we set $v_B=n$'}
%\m{made $Q_B$ a set}
Let $Q_B = V(B) \setminus \{v_B\}$.  
Note that every vertex other than $n$ appears in exactly one set $Q_B$, so the $k$ sets $Q_B$ partition $[n-1]$.

\item
Relabel the blocks as $B_1,\ldots,B_k$ in some canonical order (say in increasing order of the largest vertex in $Q_B$).
For each $i =1,\ldots,k$, denote $Q_{B_i}$ by $Q_i$ and $v_{B_i}$ by $v_i$
(the vertices $v_i$ need not be distinct).  Additionally, set $Q_{k+1}=\{n\}$.

\item
For each $i=1,\dots,k$, add a new `ghost' vertex $g_i$ to $Q_i$, and add edges from $g_i$ to the neighbours of $v_i$ in $B_i$.
%\m{was `neighbours of $v_i$ in $Q_i$'}
We set $P_i= Q_i \cup \{g_i\}$, and set $P_{k+1}=\{n\}$.  Thus the $P_i$ partition $V(G)\cup\{g_1,\dots,g_k\}$ 
and, for $i=1,\dots k$, $P_i$ induces a copy of~$B_i$.

\item
Finally, we delete all edges between the sets $P_i$, and then add edges $g_iv_i$ for each $i$.
\end{itemize}
Let $\widetilde G$ be the resulting graph, with vertex partition $\cp_G=\{P_1,\dots,P_{k+1}\}$.

The edges $g_iv_i$ join up the $P_i$ in a tree structure encoding the block structure of $G$.  
Observe that each edge $g_i v_i$ is a bridge in $\widetilde{G}$, and if we contract each of the edges $g_i v_i$   we obtain the original graph $G$. If we start with $\widetilde{G}$ and contract each set $P_i$ to a single node $i$ then we form a tree on $[k+1]$, which we denote by $T_G$.
Note that if $G$ has a path with edges in $t+1$ distinct blocks then $T_G$ has a path of length $t$ (as edges in distinct blocks correspond to edges in distinct sets $P_i$, which are contracted into distinct vertices of $T_G$).

Let $\cC$ be a weakly block-stable class of connected graphs. 
%Let %$\ca$ be a block class of graphs and let $\cC$ be the class of connected graphs in a weakly block-stable class.
Let $G$ be a (fixed) graph in $\cC_n$, and suppose that $G$ has $k \geq 2$ blocks.
Let $\cP_G=\{P_1,\ldots,P_{k+1}\}$.
%\m{dropped $\cg$}
Let the \emph{explosion neighbourhood} $\cG_G$ of $G$ be the set of all connected graphs $H$ on $[n]$ such that $\cP_H=\cP_G$, and the induced graphs $\widetilde{H}[P_i] = \widetilde{G}[P_i]$ for each $i=1,\ldots,k$. (Note that the labelled graphs $\widetilde{H}[P_i]$ and $\widetilde{G}[P_i]$ are identical, not just isomorphic.) Then for each graph $H$ in $\cG_G$, the blocks of $G$ and $H$ are the same up to isomorphism (although they may be attached to each other differently); and thus $H \sim G$, $H$ is in $\cC_n$ and $\cG_G \subseteq [G] \subseteq \cC_n$.   Thus $[G]$ is partitioned into disjoint explosion neighbourhoods.
Also, notice that if $H$ is in $\cG_G$ and the trees $T_H$ and $T_G$ are the same, then the only differences between $\widetilde{H}$ and $\widetilde{G}$ are the choices of `external' neighbours for the ghost vertices $g_i$.
Recall that $v_i$ is the neighbour of $g_i$ outside $P_i$ in $\widetilde{G}$. If $v_i$ is in $P_j$  (and so $v_i$ is in $Q_j$) %\m{so we need to define $Q_{k+1}$}
%($v_i$ is the neighbour of $g_i$ outside $P_i$ in $\widetilde{G}$)
then in $\widetilde{H}$ we may have any vertex in $Q_j$ as neighbour of $g_i$. 
%\m{$Q_j$, not $P_j$}
%\m{`Should we comment that: given $P_1,\dots,P_{k+1}$, we can join each ghost vertex to any vertex provided
%it is (a) not another ghost vertex and (b) closer to $n$ in the resulting graph?  The Pr\"ufer construction is then OK because $n$ has the largest label.'}
%\m{I have had a go}

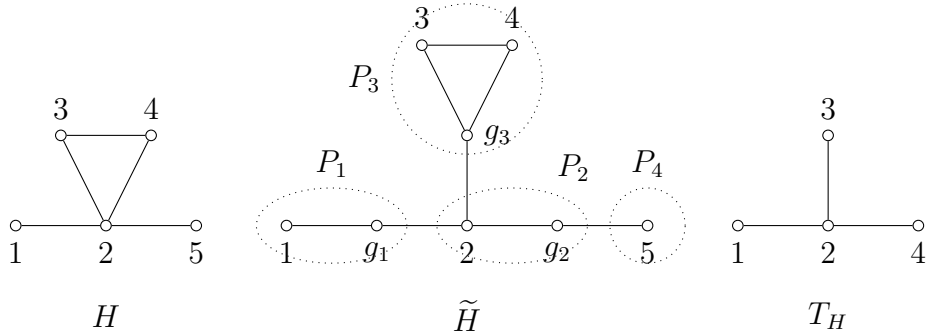
\begin{figure}
\begin{tikzpicture}
  [scale=.6,auto=left,
vx/.style={circle,draw,minimum size=4pt,inner sep=0pt},
capt/.style={rectangle,minimum size=4pt,inner sep=0pt}]

\node (n1) at (1,1) [vx, label=below:1] {};
\node (n2) at (3,1) [vx, label=below:2] {};
\node (n3) at (2,3) [vx, label=above:3] {};
\node (n4) at (4,3) [vx, label=above:4] {};
\node (n5) at (5,1) [vx, label=below:5] {};

\node at (3,-1) [capt] {$H$}; 

\foreach \from/\to in {n1/n2,n2/n3,n3/n4,n4/n2,n2/n5}
\draw (\from) -- (\to);

\node (o1) at (7,1)  [vx, label=below:1] {};
\node (og1) at (9,1)  [vx, label=below:$g_1$] {};

\node (o2) at (11,1)  [vx, label=below:2] {};
\node (og2) at (13,1)  [vx, label=below:$g_2$] {};

\node (og3) at (11,3)  [vx, label=right:$g_3$] {};
\node (o3) at (10,5)  [vx, label=above:3] {};
\node (o4) at (12,5)  [vx, label=above:4] {};

\node (o5) at (15,1)  [vx, label=below:5] {};

\draw[dotted] (8,1) 
node[ellipse, minimum height=1cm,minimum width=2cm,draw,label=above:$P_1$] {};
\draw[dotted] (12,1) 
node[ellipse, minimum height=1cm,minimum width=2cm,draw,label=above right:$P_2$] {};
\draw[dotted] (11,4.25) 

node[ellipse, minimum height=2cm,minimum width=2cm,draw,label=left:$P_3$] {};
\draw[dotted] (15,1) 
node[ellipse, minimum height=1cm,minimum width=1cm,draw,label=above:$P_4$] {};

\node at (11,-1) [capt] {$\widetilde H$}; 

\foreach \from/\to in {o1/og1,og1/o2,o2/og2,og2/o5,o2/og3,og3/o3,o3/o4,o4/og3}
\draw (\from) -- (\to);

\node (p1) at (17,1)  [vx, label=below:1] {};
\node (p2) at (19,1)  [vx, label=below:2] {};
\node (p3) at (19,3)  [vx, label=above:3] {};
\node (p4) at (21,1)  [vx, label=below:4] {};

\node at (19,-1) [capt] {$T_H$}; 

\foreach \from/\to in {p1/p2,p2/p3,p2/p4}
\draw (\from) -- (\to);

\end{tikzpicture}

\caption{{\em Construction of $\widetilde H$ and $T_H$ from $H$.  
The graph $H$ has three blocks, giving $Q_1=\{1\}$,  $Q_2=\{2\}$, $Q_3=\{3,4\}$
and $v_1=v_3=2$, $v_2=5$.  Note that in $\widetilde H$ the ghost vertices $g_1$ and $g_3$ are clones of 2, and $g_2$ is a clone of 5.}
}
\end{figure}

Further, suppose that we start from $\cP_G=\{P_1,\ldots,P_{k+1}\}$, and construct a graph $K$ as follows. For each $i=1,\ldots,k$ we choose any neighbour $u_i$ for $g_i$ such that
\begin{itemize}
  \item $u_i \in P_j$ for some $j \neq i$ and $u_i$ is not a ghost vertex (that is, `$u_i \in Q_j$'), and
  \item the graph obtained from $K$ by contracting each $P_i$ to a single node $i$ is a tree $T$ on $[k+1]$.
\end{itemize}
Then the graph $H$ obtained from $K$ by contracting each newly added edge $g_iu_i$ is in $\cG_G$, $K$ is the tree-like graph $\tilde{H}$ corresponding to $H$, and $T$ is the corresponding tree $T_H$.

The distinct parts $\cG_G$ for $G \in \cC_n$ partition $\cC_n$, and so it will suffice for us to fix one graph $G \in \cC_n$ where $G$ has $k \geq 2$ blocks, and consider the part $\cG_G$. %\m{was `$\cg=\cg_G$'}
We shall see that there is a natural bijection between $\cG_G$ and $[n]^{k-1}$, obtained by a slight extension of Pr$\ddot{\rm u}$fer coding.
Recall that the Pr\"ufer coding of a labelled tree $T$ is obtained by repeatedly deleting the leaf with smallest label and recording the label of its neighbour, repeating until two vertices remain; this gives a bijection between trees on $[n]$ and elements of $[n]^{n-2}$.
Given a tree $T$ on $[n]$ for some $n \geq 2$ let ${\bf t} = {\bf t}(T) \in [n]^{n-2}$ denote its
%\m{ref for Pr$\ddot{\rm u}$fer?}
Pr$\ddot{\rm u}$fer codeword; and given ${\bf t} \in [n]^{n-2}$ let $T=T({\bf t})$ be the corresponding tree.

For a graph $H \in \cg_G$, let us consider the tree-like graph $\widetilde{H}$ and the tree $T_H$ on $[k+1]$.  We construct a codeword ${\bf x}_H= (x_1,\ldots,x_{k-1}) \in [n]^{k-1}$ as follows:
if $i$ is the leaf of $T_H$ with smallest label, and $j$ is the neighbour of $i$ in $T_H$, then let $x_1$ be the neighbour 
%\m{was $x_i$}
of $g_i$ in $P_j$, record $x_1$, and delete vertex $i$; repeat to find $x_2$ (if $k \geq 3$), and continue until two vertices remain.  In the example in Figure 1, ${\bf x}_H=(x_1,x_2)=(2,2)$. 
Further, let $f:[n] \to [k+1]$ be given by setting $f(i)=j$ if vertex $i$ is in $P_j$.
If ${\bf x}_H=(x_1,\ldots,x_{k-1})$ then the Pr$\ddot{\rm u}$fer codeword ${\bf t}(T_H)$ is  $(f(x_1),\ldots,f(x_{k-1}))$.

Just as Pr$\ddot{\rm u}$fer coding gives a bijection between trees on $[k+1]$ and vectors in $[k+1]^{k-1}$, so the map $: H \to {\bf x}_H$ gives a bijection as required, between $\cg_G$ and $[n]^{k-1}$.  Also, for each vertex $j$ of $H$, the number of blocks of $H$ containing $j$ is $1+ a(j,{\bf x}_H)$, where $a(j,{\bf x})$ is the number of \emph{appearances} of $j$ in the vector $\bf x$, that is, the number of co-ordinates of ${\bf x}$ which are equal to $j$.

For each $j=1,\ldots,k$ let $w_j = |f^{-1}(j)|= |Q_j|=|P_j|-1$, 
%  \m{note that $\sum w_i=n$? done}
and let $w_{k+1}=1$ 
%  \m{was `(thus, for  $j=1,\dots,k$, $w_j$ is the number of choices for the neighbour $v_i$ of $g_i$ in $P_i$'}
(thus, for $j=1,\dots,k+1$, $w_j$ is the number of choices for the neighbour $v_i$ of $g_i$ in $Q_j$, and $\sum_j w_j=n$).
Let $T$ be a tree on $[k+1]$, with corresponding codeword ${\bf t}=(t_1,\ldots,t_{k-1}) \in [k+1]^{k-1}$.
Then by the above the number of graphs $H \in \cg_G$ with $T_H=T$ is
\begin{equation} \label{eqn.count}
  \prod_{i=1}^{k-1}w_{t_i} = \prod_{j=1}^{k+1} w_j^{a(j,{\bf t})} = \prod_{j=1}^{k+1} w_j^{d_T(j)-1}.
\end{equation}

Let $n \geq 3$. Consider a connected graph $G$ with vertex set $V=[n]$ and with $k \geq 2$ blocks, and with corresponding explosion neighbourhood $\cg_G$ as above. Let $R \inu \cg_G$.  Consider the corresponding tree-like graph $\widetilde{R}$ and tree $T_R$.  We shall identify the distributions of the extended Pr$\ddot{\rm u}$fer codeword ${\bf x}_R \in [n]^{k-1}$ corresponding to $\widetilde{R}$, and of the Pr$\ddot{\rm u}$fer codeword ${\bf t}(T_R) \in [k+1]^{k-1}$ corresponding to $T_R$.

For ${\bf x}_R$ this is easy: we have already noted that there is a bijection between the graphs in $\cg_G$ and the codewords, and so ${\bf x}_R$ is uniformly distributed over $[n]^{k-1}$.
For ${\bf t}(T_R)$,
%Let $n \geq 3$ and let $w_1,\ldots,w_n >0$.
let the random variable $X$ take values in $[k+1]$, with $\pr(X=j)=w_j/\sum_{i=1}^{k+1} w_i$, and let ${\bf X}=(X_1,\ldots,X_{k-1})$ where $X_1,\ldots,X_{k-1}$ are independent, each distributed like $X$. Then ${\bf X}$ and ${\bf t}(T_R)$ have the same distribution.  For, by~(\ref{eqn.count}), given a vector ${\bf t}=(t_1\ldots,t_{k-1}) \in [k+1]^{k-1}$, both $\pr({\bf X}={\bf t})$ and $\pr({\bf t}(T_R)={\bf t})$ are proportional to $\prod_{j=1}^{k+1} w_j^{a(j,{\bf t})}$, and they both take the same values ${\bf t}$ so the normalising constants must be the same.

%%%%%%%%%%%%%%%%%%%%%%%%%%%%%%%%%%%%%%%%%%%%%%%%%%%%%%%%%%%%%%%%%%%%%%%%%%%%%%%%%%%%%%%

%%%%%%%%%%%%%%%%%%%%%%%%%%%%%%%%%%%%%%%%%%%%%%%%%%%%%%%%%%%%%%%%%%%%%%%%%%%%%%%%%%%%%%%

\section{Number of blocks containing a vertex}
\label{sec.blockdeg}

%We start with the proof that the block degree sequence of a random connected graph from a block  class is dominated by the degree sequence of a random tree.

We begin by showing that, for any weakly block-stable class of connected graphs, the block degree sequence of a random graph in the class is stochastically dominated by the degree sequence of a random tree.
% (even though we do not in fact need this fact in our proofs).
%The key part of the next lemma for our proofs is the inequality~(\ref{eqn.maxdeg0}). %\m{or \eqref{eqn.maxdeg0}?}

\begin{lemma} \label{lem.conn-degree}
Let $\cc$ be a weakly block-stable class of connected graphs,
%such that a connected graph $G$ is in $\cc$ if and only if each block of $G$ is in $\cc$.  Suppose that $\cc_n$ is not empty,
and let $R_n \inu \cc$. Then  $(\widetilde{d}_{R_n}(v): v \in [n])$ is stochastically at most $(d_{T_n}(v):v \in [n])$, where $T_n \inu \ct$ is a uniformly random tree on~$[n]$.

Indeed, let $G$ be a fixed graph in $\cc_n$ with $k$ blocks, let $\cg_G$ be the explosion neighbourhood of $G$, and let $R \inu \cg_G$. Then 
$(\widetilde{d}_{R}(v): v \in [n])$ is stochastically at most $(d_{T_n}(v):v \in [n])$. Further, %if $G$ has $k$ blocks, then for real $s \geq 1$, %$s>0$
\begin{equation} \label{eqn.maxdeg0}
  \pr( \widetilde{\Delta}(R) \geq s+1) \leq n \left(\frac{ek}{ns}\right)^s \leq k (e/s)^s.
\end{equation}
\end{lemma}

\begin{proof}
  It suffices to prove the statements concerning $R \inu \cg_G$. %lemma with $R_n$ replaced by $R$.
  Recall that ${\bf x}_R \inu [n]^{k-1}$. The block-degrees $\widetilde{d}_R(j)$ of the vertices $j$ of $R$ satisfy
\begin{equation} \label{eqn.deg3}
 ( \widetilde{d}_R(1),\ldots,\widetilde{d}_R(n)) = (a(1,{\bf x}_R)+1,\ldots,a(n, {\bf x}_R)+1).
\end{equation}
Now let $T \inu \ct_n$.  Recall that ${\bf t}={\bf t}(T) \inu [n]^{n-2}$, and
%\begin{equation} \label{eqn.deg2}
\[ ( d_T(1),\ldots,d_T(n)) = (a(1,{\bf t})+1,\ldots,a(n, {\bf t})+1).\]
%\end{equation}
But  $k-1 \leq n-2$, and so $(\widetilde{d}_{R}(v): v \in [n])$ is stochastically at most $(d_{T_n}(v):v \in [n])$.
%the lemma follows immediately.
%The case of the lemma for $R \inu \cg$ now follows immediately,
%from equations~(\ref{eqn.deg1}) and~(\ref{eqn.deg2}) above, since $k-1 \leq n-2$.
Also, by~(\ref{eqn.deg3}), for each integer $s > 0$,
\begin{eqnarray*}
  \pr(\widetilde{d}_R(1) \geq s+1) &=& \pr(a(1,{\bf x}_R) \geq s) = \pr(\bin(k-1,n^{-1})\geq s)\\
  & \leq & \binom{k-1}{s} n^{-s} \leq \left(\frac{ek}{ns}\right)^s.
\end{eqnarray*}
  Thus for each integer $s \geq 1$
\[ \pr( \widetilde{\Delta}(R) \geq s+1) \leq n (ek/ns)^s \leq k(e/s)^s. \]
%as required.
Finally, since $(e/x)^x$ is decreasing in $x$ for $x \geq 1$ we may drop the assumption that $s$ is integral, to obtain~(\ref{eqn.maxdeg0}).
\end{proof}

Lemma~\ref{lem.conn-degree} proves part (a) of Theorem~\ref{thm.degreesnew}.  The next lemma is a more detailed version of part (b) of that theorem, and will quickly yield that result.

\begin{lemma} \label{lem.maxdegreenew}
 Let $\ca$ be a weakly block-stable class of graphs. Fix a graph $G \in \cA_n$, %with components $G_1,\ldots,G_j$ for some $j \geq 1$, and 
 with a total of $k$ blocks.  
 %Let $\cB(G)$ be the set of graphs in $\cA_n$ which have $j$ components $H_1,\ldots,H_j$  where $V(H_i)=V(G_i)$ and $H_i \sim G_i$ for each $i$ (so $H \sim G$).
 % has the same (up to isomorphism) list of $k_i$ blocks as $G_i$ (so $\sum_ik_i=k$). %Let $R^{G}_n \inu \cb(G)$.
Let $R_n \inu \cA$.
Then, for each real $s \geq 1$ we have
\begin{equation} \label{eqn.maxdeg1}
  \pr( \widetilde{\Delta}(R_{n}) \geq s+ 1 \, | \, R_n \in [G]) \leq k (e/s)^s. %n (er/s)^s  \leq n (e/s)^s.
\end{equation}
\end{lemma}

Life would have been tidier if there had been a detailed stochastic dominance result here corresponding to that in Lemma~\ref{lem.conn-degree} involving a random tree - but unfortunately that is not the case.  For example, let $\ca$ be the class of forests, let $n=6$, let $G \in \ca_6$ have two components both of which are paths of length 2, and let $R_n \inu [G]$.
Let $\cp$ be the increasing set in $\{0,1,\ldots\}^6$ where ${\bf x} \in \cp$ when we can partition the set $[6]$ of co-ordinates into two 3-sets $I$ and $J$ such that both $\sum_{i \in I} x_i \geq 4$ and $\sum_{j \in J} x_j \geq 4$.
Then the probability that the (block) degree sequence of $R_n$ is in $\cp$ is 1, but the probability that the degree sequence of $T_n$ is in $\cp$ is $<1$,
%$\pr(R_n \in \cp)=1$ but $\pr(T_n \in \cp)<1$ for $T_6 \inu \ct$, 
since for example $T_n$ can be a star.
Thus here (with $n=6$) it is not true that $(d_{R_n}(v): v \in [n])$ is stochastically at most $(d_{T_n}(v): v \in [n])$.

However, we can use the stochastic dominance in Lemma~\ref{lem.conn-degree} `component by component'.

%\m{the next paragraph is to be suppressed}
%We do not know if it is true for $R_n \inu \ca$ (not conditioned) that
%$(\widetilde{d}_{R_n}(v): v \in [n])$ is stochastically at most $(d_{T_n}(v):v \in [n])$?
%Indeed, suppose that $\cf$ is the class of forests, $n$ is a positive integer, and $R_n \inu \cf$.  Then a natural question is whether
%$(d_{R_n}(v): v \in [n])$ is stochastically at most $(d_{T_n}(v):v \in [n])$.
%This a weakening of the following question.  Can we couple $R_n \inu \cf$ and $T_n \inu \ct$ so that always $R_n$ is a subgraph of $T_n$?  Equivalently, is there a function $f: \cf_n \times \ct_n \to {\mathbb R}^+$ such that
%$\sum_{F \in \cf} f(F,T)= \frac{|\cf_n|}{| \ct_n|}$ for each $T \in \ct_n$ and 
%$\sum_{T \in \ct} f(F,T)= \frac{|\ct_n|}{| \cf_n|}$ for each $F \in \cf_n$?
%Reference?  Is the answer known? Known for rooted forests and trees?  Pitman?

\begin{proof}[Proof of Lemma~\ref{lem.maxdegreenew}]

Let $\cC$ be the class of connected graphs in $\cA$.
Suppose that the graph $G \in \cA_n$ has components $G_1,\ldots,G_j$ for some $j \geq 1$.  For each $i=1,\ldots,j$ let $W_i=V(G_i)$.
% and let $\cD_i$ be the set of graphs in $ [G_i]$ (each on vertex set $W_i$).
% with the same list of $k_i$ blocks as $G_i$ (so $\cD \subseteq [G_i]$ these graphs are in $\cC$).
Suppose that $G_i$ has $k_i$ blocks, and observe that $\sum_ik_i =k$.  Let $\cE$ be the set of graphs on $[n]$ with no edges between distinct sets $W_i$ and $W_{i'}$. Then for a graph $H$ on $[n]$, $H \in [G]$ iff $H \in \cE$ and $H[W_i] \in [G_i]$ for each $i$.

For each $i$, let the random graph $S_{i}$ be uniformly distributed over the graphs in $[G_i]$. % on vertex set $[n_i]$. with the same list of blocks as $G_i$ (these graphs are in $\cc$).
Then
\begin{align*}
    \pr(\widetilde{\Delta}(R_n[W_i]&) \geq s\!+\!1 \, |\, R_n \in [G]))\\
  & = 
    \pr(\widetilde{\Delta}(R_n[W_i]) \geq s\!+\!1 \,|\, \{R_n[W_i] \in [G_i]\} \cap \{R_n \in \ce\})\\
  & = 
    \pr(\widetilde{\Delta}(S_i) \geq s\!+\!1) \;\;\; \leq \;\;  k_i (e/s)^s
\end{align*} 
  by Lemma~\ref{lem.conn-degree}.
Hence, by the union bound %and Lemma~\ref{lem.conn-degree} 
\begin{align*}
    \pr(\widetilde{\Delta}(R_n) \geq s\!+\!1 \, |\, R_n \in [G])
% &=&
%    \pr(\widetilde{\Delta}(R_n) \geq s\!+\!1 \, |\, R_n \in \cb') \\
  & \leq 
    \sum_{i=1}^{j} \pr(\widetilde{\Delta}(R_n[W_i]) \geq s\!+\!1 \,|\, R_n \in [G]) \\
  %& \leq &
  %\sum_{i=1}^{j} \pr(\widetilde{\Delta}(R_n[W_i]) \geq s\!+\!1 \,|\, R_n[W_i] \in \cc_{n_i} \mbox{ with blocks } B_i^1,\ldots,B_i^{k_i} )\\
     & \leq 
  \sum_{i=1}^{j} k_i (e/s)^s % \left( \frac{ek_i}{n_i s} \right)^s 
     \;\; = \;\; k (e/s)^s %\left( \frac{er}{s} \right)^s,
\end{align*} 
as required.
% Now we may remove the conditioning to complete the proof.
%The result for $R_n \inu \ca$ follows by conditioning on the vertex sets of the components of $R_n$.
\end{proof}  
\medskip

We may now complete the proof of Theorem~\ref{thm.degreesnew}.

\begin{proof}[Proof of Theorem~\ref{thm.degreesnew}]
It remains to prove %suffices to establish the last sentence of 
part (b) of the theorem.  Let $G \in \cA_n$ have %$j$ components and 
$k$ blocks.  It suffices to show that~(\ref{eqn.degk}) (and thus~(\ref{eqn.deg1})) and~(\ref{eqn.deg2}) hold for $R_n$ conditioned on $R_n \in [G]$.
%
%We shall show that for each $s \geq 1$ we have
%\begin{equation} \label{eqn.maxdeg2}
% \pr[ \widetilde{\Delta}(R_{n}) \geq s+1 \, | \, R_n \in [G]] \leq k (e/s)^s.
%\end{equation}
%To see this, note that $[G]$ may be partitioned into sets $\cB(H_1), \cB(H_2),\ldots$ as in the last lemma, for some graphs $H_i$ in $[G]$. Hence the left hand side above is a weighted average over graphs $H \in [G]$ % $H \in \ca_n$ with $j$ components and the same list of $k$ blocks as $G$, of the values
%\[ \pr[ \widetilde{\Delta}(R_{n}) \geq s+ 1 \, | \, R_n \in \cB(H)]. \]
%But by the last lemma each of these values is at most the right hand side of~(\ref{eqn.maxdeg2}), and~(\ref{eqn.maxdeg2}) follows.
To see this, we use the last lemma: set $s+1= (1+\eps) \log k/ \log\log k$ to deduce~(\ref{eqn.degk}), and $s+1=cn/\log n$ to deduce~(\ref{eqn.deg2}).
\end{proof}

%\end{document}

%%%%%%%%%%%%%%%%%%%%%%%%%%%%%%%%%%%%%%%%%%%%%%%%%%%%%%%%%%%%%%%%%%%%%%%%%%%%%%%%%

%%%%%%%%%%%%%%%%%%%%%%%%%%%%%%%%%%%%%%%%%%%%%%%%%%%%%%%%%%%%%%%%%%%%%%%%%%%%%%%%%

\section{Path lengths}
\label{sec.pathlengths}

Let $Q(t)$ denote the class of graphs $G$ which have a path containing edges in at least $t$ different blocks.  Thus a forest is in $Q(t)$ if and only if it has a path of length at least~$t$.  We first consider connected graphs.
%
%\m{adjust $t$ or $t+1$,..?}
\begin{lemma} \label{lem.conn-path}
  Let $\cC$ be a weakly block-stable class of connected graphs, let $G$ be a fixed graph in $\cC_n$ with $k$ components, let $\cG_G$ be the explosion neighbourhood of $G$, and let $R \inu \cG_G$. Then for each $t \geq 0$, 
  \begin{equation} \label{eqn.qboundc}
    \pr(R \in Q(t+2)) \leq 2 k^2 e^{-\frac{t^2}{2(k+1)}}. %\binom{k+1}{2} e^{-\frac{t^2}{k+1} }. % n^2 e^{-\frac{t^2}{2n} }.
  \end{equation}
\end{lemma}
%\noindent
%  Lemma~\ref{lem.conn-path} will easily yield:
  %
%\begin{lemma} \label{lem.gen}
%  Let $\ca$ be a block class of graphs,
%  % such that $G \in \ca$ if and only if each block of $G$ is in $\ca$; 
%  %Suppose that $\ca_n$ is non-empty 
%  and let $R_n \inu \ca$.  Then for each $t \geq 0$, %\m{was `positive integer $t$'}
%  \begin{equation} \label{eqn.qbound}
%    \pr[R_n \in Q(t+2)] \leq n^2 \ e^{-\frac{t^2}{2n}}.
%  \end{equation}
%\end{lemma}
  %
%  Fix a positive integer $b$, and suppose that each 2-connected graph in $\ca$ has order at most $b$.    If a graph %$G$ is not in $Q(t)$ then the maximum length of a path is at most $t(b-1)$.
%  Thus we obtain:
  %
%\begin{corollary}
%  Let $b$ be a fixed positive integer.  Then for each $\eps>0$ there exists $\delta>0$ such that,
%  for each non-trivial block class of graphs with each block of order at most $b$,
%  the probability that $R_n$ has a path of length at least $\eps n$ is at most $e^{-\delta n}$.
%\end{corollary}

In order to prove Lemma~\ref{lem.conn-path} we need two lemmas: the first preliminary lemma may well be known but we give a proof for completeness.
\begin{lemma} \label{lem.distinct}
  Let $2 \leq j \leq n$, and let $X_1,\ldots,X_j$ be iid random variables taking values in $[n]$.
Then (a) the probability that $X_1$ is not repeated is at most $(1-1/n)^{j-1}$; 
and (b) the probability that $X_1,\ldots,X_j$ are all distinct is at most $(n)_j/n^j$.
\end{lemma}
\begin{proof}
  Denote $\pr(X_1 =i)$ by $p_i$ for $i=1,\ldots,n$, and set ${\bf p}=(p_1,\dots,p_n)$.

  (a) The probability that $X_1$ is not repeated is $g({\bf p}) = \sum_{i=1}^n p_i(1-p_i)^{j-1}$.
%f(p_i)$ where $f(x)=x(1-x)^{j-1}$ for $0<x<1$.
Suppose first that $j=2$.  Then $g({\bf p}) = 1- \sum_{i=1}^n p_i^2 \leq 1- 1/n$ since, as is well known, $\sum_{i=1}^n p_i^2$ is minimised when each $p_i=1/n$.  

Now suppose that $j \geq 3$.  Let $f(x)=x(1-x)^{j-1}$ for $0\leq x \leq 1$. Then $g({\bf p}) = \sum_{i=1}^n f(p_i)$. Let $m$ be the maximum value of this quantity, achieved at ${\bf q} = (q_1,\ldots,q_n)$.  Now $f'(x)= (1-x)^{j-2} (1-jx)$, which is $>0$ for $0<x<1/j$,
 $=0$ at $x=1/j$ and $<0$ for $1/j<x<1$.  Also $f''(x) = (j-1)(1-x)^{j-3}(2-jx)$, which is $>0$ for $0<x<2/j$.

Clearly each $q_i \in [0,1)$.
%  Since the continuous function $f(x)$ is increasing and strictly concave on $[0,1/j]$
%  and decreasing on $[1/j,1]$ it follows that each $q_i >0$.
%  For if $q_{i_0}=0$ and $q_{i_1}>0$ then we can strictly increase $g({\bf q})$ by replacing
%  $q_{i_0}$ and $q_{i_1}$ by two co-ordinates equal to their average. 
%  %
%  Now let $L = g({\bf x}) - \lambda \sum_{i=1}^n x_i$.  Differentiating this Lagrangian
%  with respect to $x_i$ gives $f'(q_i)= \lambda$ for each $i$.
%  But some $q_i \leq 1/j$ so $\lambda \geq 0$.  Now
%  $f'(x)= \lambda$ has a unique solution in $[0,1)$, so each $q_i$ takes the same value,
%  which must be $1/n$.  Hence $m = (1-1/n)^{j-1}$, which completes the proof of (a).
%  \smallskip
%  
%{\bf Alternative for last para:}
If $q_i>1/j$ for some $i$ then there is $k$ with $q_k<1/j$ (as $\sum_kq_k=1$); increasing $q_i$ and decreasing $q_k$ slightly would then increase $g({\mathbf q})$.  We must therefore have $\max_iq_i\le 1/j$, and so by (strict) convexity $g({\mathbf q})$ is (uniquely) maximized when all the $q_i$ take the same value, which must be $1/n$.  Hence $m = (1-1/n)^{j-1}$, which completes the proof of (a).
\smallskip

  (b)
%For $1 \leq j \leq n$ let $(n)_j/n^j$.
Consider any positive integer $n$.  The result is trivially true for $j=1$.  Let $2 \leq j \leq n$ and suppose that it holds for $j-1$.
%Let $X_1,\ldots,X_j$ be iid taking values in $[n]$, and let $p_i=\Pr[X_1=i]$.
Let $A_i$ be the event that none of $X_2,\ldots,X_j$ are equal to $i$. Then by conditioning on $X_1$ and using the induction hypothesis, we find
%the probability that $X_1,\ldots,X_j$ are all distinct is at most
\begin{align*}
  \pr (X_1,\ldots,X_j \mbox{ distinct})
  &=
  \sum_{i=1}^n \pr(X_1=i, A_i) \ \pr(X_2,\ldots,X_j \mbox{ distinct} | A_i)\\
 & \leq 
  \sum_{i=1}^n \pr(X_1=i, A_i) \ \frac{(n-1)_{j-1}}{(n-1)^{j-1}}\\
 &=
   \pr(X_1 \mbox{ not repeated }) \ \frac{(n-1)_{j-1}}{(n-1)^{j-1}}\\
 & \leq 
  \left (\frac{n-1}{n}\right)^{j-1}  \frac{(n-1)_{j-1}}{(n-1)^{j-1}} \;  = \; \frac{(n)_j}{n^j}
\end{align*}
as required.
%  \[ \sum_{i=1}^{n} p_i(1-p_i)^{j-1} \cdot (n-1)_{j-1}/(n-1)^{j-1} \leq (n)_j/n^j. \]
%  since
%  \[
%    \sum_{i=1}^{n} p_i(1-p_i)^{n-1} \leq n \max_{0 \leq x \leq 1}
%    x(1-x)^{n-1} = (1-1/n)^{n-1}.
%  \]
\end{proof}
\medskip

\begin{lemma} \label{lem.pathlength}
Let $m \geq 3$ and let $w_1,\ldots,w_m >0$. Let the random variable $X$ take values in $[m]$, with $\pr(X=j)=w_j/\sum_{i=1}^{m} w_i$. Let ${\bf X}=(X_1,\ldots,X_{m-2})$ where $X_1,\ldots,X_{m-2}$ are independent, each distributed like $X$.
Consider the random tree $T({\bf X})$ on $[m]$.
For each integer $t \geq 1$, the expected number of paths of length at least $t+1$ is at most 
\[ 
%  \E[\mbox{number of paths of length}\geq t+1 ] \leq 
\binom{m}{2} e^{-\binom{t}{2}/m} \leq 2(m-1)^2 e^{-\frac{t^2}{2m}}.
\]
%\m{or $\leq m^2 e^{-\frac{t^2}{2m}}$}
\end{lemma}

\noindent
Before we prove this lemma, let us note that it will yield Lemma~\ref{lem.conn-path}. To see this, let $G$ have $k$ blocks, and set $m=k+1$ in the last lemma.
%For we may consider $R \inu \cg$ as before, and 
Now recall from the end of Section~\ref{sec.Pcoding} that, for a suitable choice of $w_1,\ldots,w_m$, $T_R$ has the same distribution as $T({\bf X})$.  But if $H \in Q(t+2)$ then $T_H$ has a path of length $t+1$.
\medskip

\begin{proof}
We first consider the distance in $T({\bf x})$ between vertices $m-1$ and $m$ using Pr\"ufer coding, and then extend to all pairs of vertices.  Given a tree $T \in \ct_m$ and distinct vertices $i,j \in [m]$, denote the distance between $i$ and $j$ in $T$ by $\dist(i,j;T)$.
We claim that
\begin{equation}\label{claim.m}
\pr(\dist(m\!-\!1,m;T({\bf X})) \geq t+1) \leq e^{-\binom{t}{2}/m}. %e^{-\frac{t^2}{2m}}.
\end{equation}
To prove this, consider any vector ${\bf x} \in [m]^{m-2}$.
If the path between vertices $m-1$ and $m$ in $T({\bf x})$ has length at least $t+1$ then the last $t$ co-ordinates of $\bf x$ are distinct (this follows by considering the algorithm for Pr\"ufer   applied to a tree $T$ on $[m]$: running the algorithm for as long as it removes leaves with labels from $[m-2]$, we are left with the path from $m-1$ to $m$ in $T$; the remaining $t$ steps of the algorithm run through the path, starting from the $m-1$ end,  and record the internal vertices of the path in order).
%(and the preceding co-ordinate is one of these values or is $n-1$ or $n$).
Hence the probability that the path between vertices $m-1$ and $m$ in $T({\bf X})$ has length at least $t+1$ is at most the probability that the last $t$ of the $X_i$ are distinct,
%multiplied by $(j+2)/n$,
which is at most $(m)_{t}/m^{t}$ by Lemma~\ref{lem.distinct}.  But
\[ (m)_{t}/m^{t} = \prod_{i=0}^{t-1} (1- i/m) \leq \exp  \left(-\sum_{i=0}^{t-1} i/m\right)
  = e^{-\binom{t}{2}/m}. \] % \exp \left(- \binom{t}{2}/m\right),\]
which establishes the claim~(\ref{claim.m}).  
%\smallskip 
 
There is sufficient symmetry for us to be able to use~(\ref{claim.m}) to show that the same bound holds for the distance between an arbitrary pair of vertices. 
%Considerations of symmetry should make this plausible, but
We spell this out now.

Let $\pi$ be a permutation of $[m]$.  We denote the image of an element $i \in [m]$ by %either $\pi(i)$ or 
$i^{\pi}$.  Given a vector ${\bf z}=(z_1,\ldots,z_m)$ %(z_i:i \in [m])$ 
let ${\bf z}^{\pi}$ denote the permuted vector with $(z^{\pi})_i=z_{\pi(i)}$.  Given a tree $T \in \ct_m$, let $T^{\pi}$ denote the tree in $\ct_m$ with an edge $i^{\pi}j^{\pi}$ for each edge $ij$ in $T$, so that $\pi$ is an isomorphism from $T$ to $T^{\pi}$.  Also, given a tree $T \in \ct_m$, let ${\bf d}(T)$ be the degree sequence $(d_T(1),\ldots,d_T(m))$ of $T$.  Thus ${\bf d}(T^{\pi}) = {\bf d}(T)^{\pi^{-1}}$.

Consider distinct vertices $i$ and $j$ in $[m]$.  Let $\pi$ be a (fixed) permutation of $[m]$ with $i^{\pi}=m\!-\!1$ and $j^{\pi}=m$.
Let ${\bf z}=(z_1,\ldots,z_m)$ be a vector of positive integers with $\sum_i z_i=2m-2$.
The permutation $\pi$ yields a bijection $\phi$ from $\{T \in \ct_m: {\bf d}(T)={\bf z}\}$ to $\{T \in \ct_m: {\bf d}(T)={\bf z}^{\pi^{-1}}\}$ which takes
$\{T \in \ct_m: {\bf d}(T)={\bf z}, \dist(i,j;T)=s\}$ to $\{T \in \ct_m: {\bf d}(T)={\bf z}^{\pi^{-1}}, \dist(m\!-\!1,m;T)=s\}$.
Also, ${\bf d}(T({\bf x}))={\bf z}$ iff the number $a(v,{\bf x})$ of appearances of $v$ in ${\bf x}$ is $z_v-1$ for each $v \in [m]$, so conditional on  ${\bf d}(T({\bf X}))={\bf z}$ all trees $T$ with ${\bf d}(T)={\bf z}$ are equally likely, with probability $\left( |\{T \in \ct_m: {\bf d}(T)={\bf z} \}| \right)^{-1}$.  Let $Y_i=(X_i)^{\pi}$ for each $i$, and let ${\bf Y}=(Y_1,\ldots,Y_m)$.  Then
\begin{eqnarray*}
  &&
  \pr(\dist(i,j;T({\bf X}))=s \, | \, {\bf d}(T({\bf X}))={\bf z})\\
  %&=&
  %\pr(\dist(i,j;T)=s | {\bf d}(T)={\bf z})\\
  &=&
  \frac{|\{T \in \ct_m: {\bf d}(T)={\bf z}, \dist(i,j;T)=s)| }
  { |\{T \in \ct_m: {\bf d}(T)={\bf z} \}|}\\
   &=&
  \frac{|\{T \in \ct_m: {\bf d}(T)={\bf z}^{\pi^{-1}}, \dist(m\!-\!1,m;T)=s)| }
  { |\{T \in \ct_m: {\bf d}(T)={\bf z}^{\pi^{-1}} \}|}\\
   &=&
  \pr(\dist(m\!-\!1,m;T({\bf Y}))=s \, | \, {\bf d}(T({\bf Y}))={\bf z}^{\pi^{-1}}).
\end{eqnarray*}
Hence, summing over the possible degree sequences ${\bf z}$,
\begin{eqnarray*}
  && \pr(\dist(i,j;T({\bf X}))=s)\\
  &=&
  \sum_{{\bf z}} \pr(\dist(i,j;T({\bf X}))=s \, | \, {\bf d}(T({\bf X}))={\bf z})
  \pr({\bf d}(T({\bf X}))={\bf z})\\
   &=&
  \sum_{{\bf z}}
  \pr(\dist(m\!-\!1,m;T({\bf Y}))=s \, | \, {\bf d}(T({\bf Y}))={\bf z}^{\pi^{-1}})
  \pr({\bf d}(T({\bf Y}))={\bf z}^{\pi^{-1}})\\
  &=&
  \pr(\dist(m\!-\!1,m;T({\bf Y}))=s).
\end{eqnarray*}
Now, since ${\bf Y}$ has a distribution of the same form as ${\bf X}$, we may apply the claim~(\ref{claim.m}) to see that
\[ \pr(\dist(i,j;T({\bf X})) \geq t+1) = \pr(\dist(m\!-\!1,m;T({\bf Y})) \geq t+1) \leq e^{-\binom{t}{2}/m}. \] %e^{-\frac{t^2}{2m}}. \]
It follows that the expected number of paths in $T({\bf X})$ of length at least $t+1$ is at most
\[ \binom{m}{2} e^{-\binom{t}{2}/m} \leq (m\!-\!1)^2 e^{-\frac{t(t-1)}{2m}}  \leq 2 (m\!-\!1)^2 e^{-\frac{t^2}{2m}}, \]
%OR
%\[ \binom{m}{2} e^{-\binom{t}{2}/m} = \frac12 m(m-1) e^{\frac{t}{2m}}  e^{-\frac{t^2}{2m}} \leq m (m\!-\!1) e^{-\frac{t^2}{2m}}, \]
since we may assume that $t \leq m$ and then $e^{\frac{t}{2m}} \leq e^{\frac12} <2$.
\end{proof}
\medskip

At this point we have completed the proof of Lemma~\ref{lem.conn-path}.  The next lemma is a more detailed version of Theorem~\ref{thm.diamnew}, and will quickly yield that result.  It may be deduced from Lemma~\ref{lem.conn-path} just as Lemma~\ref{lem.maxdegreenew} was deduced from Lemma~\ref{lem.conn-degree}.

\begin{lemma}  \label{lem.pathnew}
  Let $\cA$ be a weakly block-stable class of graphs.  Fix a graph $G \in \cA_n$, with %components $G_1,\ldots,G_j$ for some $j \geq 1$, and with 
a total of $k$ blocks. %Let $\cB(G)$ be the set of graphs in $\cA_n$ which have $j$ components $H_1,\ldots,H_j$ where $V(H_i)=V(G_i)$ and $H_i \sim G_i$ for each $i$ (so $H \sim G$). 
%has the same (up to isomorphism) list of $k_i$ blocks as $G_i$  (so $\sum_ik_i=k$).
Let $R_n \inu \cA$.  Then for each $t \geq 0$, 
\begin{equation} \label{eqn.qboundd}
    \pr(R_n \in Q(t+2) \, |\, R_n \in [G]) \leq %\binom{k+1}{2} e^{-\frac{t^2}{k+1}} \leq 
    2 k^2 e^{-\frac{t^2}{2(k+1)}}. % n^2 e^{-\frac{t^2}{2n} }.
\end{equation}
\end{lemma}

%\begin{lemma} \label{lem.pathnewa}
%  Let $\ca$ be a block class of graphs, let $n$ be a positive integer, and let $R_n \inu \ca$.  % and let $\cc$ be the class of connected graphs in $\ca$.
%  Let $\cb$ be the set of graphs in $\ca_n$ with $j$ components and $k$ blocks, such that
%  (ordering the components $1,\ldots,j$ as before) %lexicographically by their vertex sets) 
%  component $i$ has $n_i$ vertices, has $k_i$ blocks, and has list of blocks $B_i^1,\ldots,B_i^{k_i}$.
%  Then for each $t \geq 0$, 
%\begin{equation} \label{eqn.qboundd}
%    \pr(R_n \in Q(t+2) \, |\, R_n \in \cb) \leq %\binom{k+1}{2} e^{-\frac{t^2}{k+1}} \leq 
%    2 k^2 e^{-\frac{t^2}{2(k+1)}}. % n^2 e^{-\frac{t^2}{2n} }.
%\end{equation}
%\end{lemma}

We may now complete the proof of  Theorem~\ref{thm.diamnew}, much as we did for Theorem~\ref{thm.degreesnew}.

\begin{proof}[Proof of Theorem~\ref{thm.diamnew}]
Let $G \in \cA_n$ have $k$ blocks.  It suffices to prove the theorem for $R_n$ conditioned on $R_n \in [G]$.
  %Let $\cB$ be the equivalence class $[G]$.
%set of graphs in $\ca_n$ with $j$ components and the same list of $k$ blocks as~$G$.  \m{what is $s$? ADS} 
If $\btf(H)$ has a path of length $t$ then $H \in Q( t/2)$. 
Thus by the last lemma %the probability that $\btf(R_n)$ has a path of length at least $a (n \log n)^{\frac12} +4$ is at most
\begin{align*}
  \pr(\btf(R_n) &\mbox{ has diameter } \geq a ((k+1) \log k)^{\frac12} +4 \, | \, R_n \in [G])\\
  %\pr(\btf(R_n) \mbox{ has a path of length } \geq a (n \log n)^{\frac12} +4 \, | \, R_n \in \cb')\\
 & \leq
   \pr(R_n \in Q(  (a/2) ((k+1) \log k)^{\frac12} +2) \, | \, R_n \in [G])\\
 & \leq 
   2 k^2 e^{-(a^2/8) \log k} \;\;\; = \; o(1) \; \mbox{ if } \; a > 4.
\end{align*}
%  if $a > 4$. 
  Further, since $k+1 \leq n$, %the probability that $\btf(R_n)$ has diameter at least $\eps n$ is at most %the probability that
  \begin{align*} 
  \pr(\btf(R_n) &\mbox{ has diameter } \geq \eps n \, | \, R_n \in [G])\\
 & \leq 
   \pr(R_n \in Q(\frac{\eps n}{3} +2) \, | \, R_n \in [G])\\
 & \leq 
   2 n^2 e^{-\frac{\eps^2 n}{18}}
\end{align*}
% \m{rearranged}
for $n \geq 12/\epsilon$  (so that  $2 (\frac{\eps n}{3} +2) \leq \eps n$).
\end{proof}  

%%%%%%%%%%%%%%%%%%%%%%%%%%%%%%%%%%%%%%%%%%%%%%%%%%%%%%%%%%%%%%%%%%%%%%%%%%%%%%%%%%%%%%%%%%%%

\section{Concluding remarks}
\label{sec.concl}

We have seen that for a random graph $R_n$ from a block-stable class,
(or from the connected graphs in such a class), 
the maximum number of blocks containing a vertex is roughly no more than for a random tree $T_n$, and the maximum number of blocks through which a path may pass is at most a factor $O(\sqrt{\log n})$ times the maximum length of a path in $T_n$.

  Let us briefly consider connectedness.  A minor-closed class is block-stable if and only if it is addable; that is, each excluded minor is 2-connected, see~\cite{cmcd09}.  Indeed, any block-stable class $\cA$ containing the single edge $K_2$ is bridge-addable, and so by~\cite{msw05} the probability that $R_n \inu \cA$ is connected is at least $1/e$, and indeed $\liminf \pr(R_n \mbox{ is connected }) \geq e^{-\frac12}$, see~\cite{amr2012,kp2013} and see also~\cite{PS2015} for a recent more general result. 
However, consider the block-stable class $\cA$ in which the only allowed block is the triangle $C_3$: the set $\cA_n$ is non-empty for each $n \geq 5$, but for each even $n$ each graph in $\cA_n$ is disconnected.
  %For related results see~\cite{mw2012+}.
  %\m{amplify?}
\bigskip

\noindent
{\bf Acknowledgement}  
  Thanks to Oliver Riordan for useful discussions, and thanks to the referees for detailed reading and helpful comments.

%%%%%%%%%%%%%%%%%%%%%%%%%%%%%%%%%%%%%%%%%%%%%%%%%%%%%%%%%%%%%%%%%%


\begin{thebibliography}{99}

\bibitem{amr2012}
  L. Addario-Berry, C. McDiarmid and B. Reed, 
   Connectivity for bridge-addable monotone graph classes,
   {\em Combinatorics, Probability and Computing} {\bf 21} (2012) 803 -- 815.

\bibitem{bps09}
  N. Bernasconi, K. Panagiotou and A. Steger,
  The degree sequence of random graphs from subcritical classes, 
  {\em Combinatorics, Probability and Computing} {\bf 18} (2009) 647 -- 681.

\bibitem{cgs94}
  R. Carr, W. Goh and E. Schmutz,
  The maximum degree in a random tree and related problems,
  {\em Random Structures and Algorithms}  {\bf 5} (1994) 13 -- 24.

\bibitem{cfgn}
  G. Chapuy, E. Fusy, O. Gim\'enez and M. Noy,
  The diameter of random planar graphs,
  {\em Combinatorics, Probability and Computing} {\bf 24} (2015) 145 -- 178.

\bibitem{drmotabook}
  M. Drmota, {\em Random Trees}, Springer Wien New York, 2009. 

\bibitem{dfkkr}
  M. Drmota, E. Fusy, M. Kang, V. Kraus and J. Ru\'e,
  Asymptotic study of subcritical graph classes,
  {\em SIAM J. Discrete Math.} {\bf 25} (2011) 1615 -- 1651.
  
\bibitem{dgnps2015}  
  M. Drmota, O. Gim´enez, M. Noy, K. Panagiotou and A. Steger,
  The maximum degree of random planar graphs,
 \emph{Proc. London Math. Soc.} (3) {\bf 109} (2014) 892 -– 920.

\bibitem{dn2013}
  M. Drmota and M. Noy,
  Extremal parameters in sub-critical graph classes, 
  Proceedings of Analco 2013, SIAM (Markus Nebel and Wojciech Szpankowki eds.) (2013) 1--7.

\bibitem{fo82}
  P. Flajolet and A. Odlyzko,
  The average height of binary trees and other simple trees,
  {\em J. Comput. System. Sci} {\bf 25} (1982) 171 -- 213.

\bibitem{fp11}
  N. Fountoulakis and K. Panagiotou,
  3-connected cores in random planar graphs,
  {\em Combinatorics, Probability and Computing} {\bf 20} (2011) 381 - 412.

\bibitem{gmn13}
  O. Gim\'enez, D. Mitsche and M. Noy,
  Maximum degree in minor-closed classes of graphs,
  {\em Europ. J. Comb.} {\bf 55} (2016) 41 -- 61.
%  arXiv: 1304.5049v1 18 April 2013. 

\bibitem{gn09a}
  O. Gim\'enez and M. Noy,
  Asymptotic enumeration and limit laws of planar graphs.
  %arXiv: math.CO/0501269, 2005  arXiv: math.CO/0501269.
  {\em J. Amer. Math. Soc.} {\bf 22} (2009) 309--329.

\bibitem{gn09b}
  O. Gim\'enez and M. Noy,
  Counting planar graphs and related families of graphs,
  in {\em Surveys in Combinatorics 2009}, 169 -- 329, Cambridge University Press, Cambridge, 2009.

\bibitem{gnr}
  O. Gim\'enez, M. Noy and J. Ru\'e,
  Graph classes with given 3-connected components,
  {\em Random Structures and Algorithms} {\bf 42} (2013) 438 -- 479.

\bibitem{k03}
  H. Kajimoto,
  An extension of the Pr\"ufer code and assembly of connected graphs from their blocks,
  {\em Graphs and Combinatorics} {\bf 19} (2003) 231--239. 

\bibitem{kp2013}
  M. Kang and K. Panagiotou,
  On the connectivity of random graphs from addable classes,
  {\em J. Comb. Theory Ser. B} {\bf 103} (2013) 306 -- 312.

%\bibitem{luczak98} \m{not wanted?}
%  T. {\L}uczak,
%  Random trees and random graphs,
%  {\em Random Structures and Algorithms} {\bf 13} (1998) 485--500.

\bibitem{cmcd09} C. McDiarmid, Random graphs from a minor-closed class.
  {\em Combinatorics, Probability and Computing} {\bf 18} (2009) 583--599.

\bibitem{mr08}
  C. McDiarmid and B. Reed, On the maximum degree of a random planar graph.
  {\em Combinatorics, Probability and Computing} {\bf 17} (2008) 591--601.

\bibitem{msw05}
  C. McDiarmid, A. Steger and D. Welsh, Random planar graphs.
  {\em J. Combin. Theory Ser. B} {\bf 93} (2005) 187--205.

%\bibitem{msw06}
%  C. McDiarmid, A. Steger and D. Welsh,
%  Random graphs from planar and other addable classes,
%  in {\em Topics in Discrete Mathematics} (M. Klazar, J. Kratochvil,
%  M. Loebl, J. Matousek, R. Thomas, P. Valtr eds),
%  {\em Algorithms and Combinatorics} {\bf 26} (2006), Springer, 231--246.

%\bibitem{mw2012+}  C. McDiarmid and K. Weller, .. ?

\bibitem{moon68}
  J.W. Moon,
  On the maximum degree in a random tree,
  {\em The Michigan Mathematical Journal} {\bf 15} (1968) 429 -- 432.

\bibitem{noy2014}
  M. Noy,
  Random planar graphs and beyond, Proceedings of the ICM, Volume IV, 407-430, Seoul 2014.

\bibitem{ps2010}
  K. Panagiotou and A. Steger,
  Maximal biconnected subgraphs of random planar graphs,
  {\em ACM Transactions on Algorithms} {\bf 6} (2010) art. no. 31
  %SODA 2009  

\bibitem{psw2015+}
  K. Panagiotou, B. Stufler and K. Weller,
  Scaling limits of random graphs from subcritical classes, arXiv:1411.1865v2 [math.PR],   {\em Ann. of Prob.}, to appear. 

\bibitem{PS2015}
G. Perarnau and G. Schaeffer,  Connectivity in bridge-addable graph classes: the McDiarmid-Steger-Welsh conjecture, arXiv: 1504.06344v1, April 2015.

\bibitem{rs67}
  A. R\'enyi and G. Szekeres,
  On the height of trees,
  {\em J. Austral. Math. Soc} {\bf 7} (1967) 497 -- 507.

\bibitem{vw01}
  J.H. van Lint and R.M. Wilson,
  {\em A Course in Combinatorics},
  CUP, 2nd ed. 2001.

\end{thebibliography}
\end{document}